\documentclass[12pt,letterpaper]{amsart}

\newtheorem {theorem}{Theorem}
\newtheorem {lemma}{Lemma}
\newtheorem{remark}{Remark}
\newtheorem {definition}{Definition}

\newtheorem {problem}{Problem}
\newtheorem {corollary}{Corollary}
\newtheorem {proposition}{Proposition}
\newtheorem {conjecture}{Conjecture}

\newtheorem{thmx}{Theorem}
  
\newtheorem{corx}{Corollary}

\newcommand {\C} {{\mathbb C}}
\newcommand {\N} {{\mathbb N}}
\newcommand {\Z} {{\mathbb Z}}
\newcommand {\R} {{\mathbb R}}

\newcommand {\TT} {{\mathbb T}}

\newcommand {\Td} {{\mathbb T}^{d}}

\usepackage{amsfonts}
\usepackage{amssymb}
\usepackage{amsmath}
\usepackage{amsthm}

\usepackage{enumerate}
\usepackage[all]{xy}
\usepackage{graphicx}
\usepackage{appendix}
 
\usepackage{url}
\usepackage{amscd}
\usepackage{texdraw}
\usepackage{graphicx}

\begin{document}

\title[Order of Oscillations]{Orders of Oscillation Motivated by Sarnak's Conjecture--Part II}
 
\author{Yunping Jiang}
\address[]{Department of Mathematics\\
 Queens College of the City University of New York\\
 Flushing, NY 11367-1597\\
 and\\
 The Ph.D. Program in Mathematics\\
 Graduate Center of the City University of New York;
 New York, NY 10016}
 \email[Jiang]{yunping.jiang@qc.cuny.edu}
 
 \keywords{oscillating sequence of order $d$, completely oscillating sequence, simple polynomial skew products, MMA and MQDS, linearly disjoint, multi-linearly disjoint}
 
\subjclass[2020]{Primary 37A35, 11K65; Secondary 37A25, 11N05}

\begin{abstract}
 This work is a continuation of~\cite{JPAMS}. We study the linear disjointness between higher-order oscillating sequences and nonlinear dynamical systems. Specifically, we prove that any oscillating sequence of order $m=d+k-1$ and any simple polynomial skew product of degree $k$ on the $d$-Euclidean space are linearly disjoint. Additionally, we demonstrate that any oscillating sequence of order $d$ and any minimal mean attractable and minimal quasi-discrete spectrum dynamical system of order $d$ are linearly disjoint. Finally, we introduce multi-linearly disjoint sequences and construct examples of such sequences.
\end{abstract}

\maketitle

\section{Introduction}

Let $\N=\{ 1, 2, \cdots, n\cdots \}$ be the set of all natural numbers. Let $X$ be a compact metric space and let $C(X)$ be the space of all continuous functions $\phi$ on $X$ equipped with the maximum norm $\|\phi \|=\max_{x\in X} |\phi (x)|$. 

Let $f: X\to X$ be a continuous map. We use $f$ to represent a (discrete) dynamical system $\{f^{n}\}_{n\in \{0\}\cup\N}$, where $f^{n}$ means the $n$-time composition of $f$ and $f^{0}$ means the identity. 

Let $\C$ be the complex plane. Let ${\bf c}=(c_{n})_{n\in \N}$ be a sequence of complex numbers in $\C$. Without loss of generality, we assume that the asymptotic mean of ${\bf c}$ is zero, that is,  
$$
{\mathbb E}({\bf c})=\lim_{N\to \infty} \frac{1}{N} \sum_{n=1}^{N} c_{n} =0.
$$
We also assume that ${\bf c}$ satisfies the following technical condition: there are constants $C>0$ and $\lambda >1$ such that 
 \begin{equation}~\label{cn}
\frac{1}{N} \sum_{n=1}^{N} |c_{n}|^{\lambda} \leq C, \quad \forall\; N\geq 1.
\end{equation}  
An example is the Liouville sequence ${\bf l} =(\lambda (n))_{n\in \N}$ where $\lambda (n) =(-1)^{\Omega (n)}$ and $\Omega (n)$ is the number of prime factors of $n$ (counted by multiplicity).  Another example is the M\"obius sequence ${\bf u} = (\mu (n))_{n\in \N}$, where $\mu (n)=0$ if $p^{2}|n$ for a prime number $p$ and $\mu (n)=\lambda (n)$, otherwise. Note that ${\mathbb E} ({\bf l})={\mathbb E} ({\bf u})=0$ (the prime number theorem). 

\medskip 
\begin{definition}~\label{lld}
We say that ${\bf c}$ and $f$ are (asymptotically) linearly disjoint 
if for any $\phi\in C(X)$,
\begin{equation}\label{ld}
  \lim_{N\to \infty} \frac{1}{N}\sum_{n=1}^N c_n \phi (f^{n}x) =0, \quad \forall\; x\in X.
\end{equation}
\end{definition}

The Sarnak conjecture in number theory~\cite{Sa1,Sa2} states

\medskip
\begin{conjecture}~\label{sc}
The sequence ${\bf u}$ and any continuous dynamical system $f: X\to X$ having zero topological entropy are linearly disjoint.
\end{conjecture}

Motivated by this conjecture, we have defined and studied oscillating sequences in~\cite{FJ}.

\medskip
\begin{definition}~\label{os}
We say that ${\bf c}$ is an oscillating sequence 
if for every $0\leq \theta <1$,
\begin{equation}~\label{hoosseq}
\lim_{N\to \infty} \frac{1}{N} \sum_{n=1}^{N} c_{n}e^{2\pi i n\theta}=0.
\end{equation} 
\end{definition} 

\medskip
\begin{definition}~\label{osas}
We say that ${\bf c}$ is an oscillating sequence in the arithmetic sense if for every $0\leq \theta <1$ and any pair of integers $0\leq l<k$,
\begin{equation}~\label{hossin}
\lim_{N\to \infty} \frac{1}{N} \sum_{1\leq n\leq N, \;\; n\equiv l \pmod{k}} c_{n}e^{2\pi i n\theta}=0.
\end{equation}
 \end{definition}
 
In addition, we define the orders of oscillation of ${\bf c}$ in~\cite{JPAMS}.  
Let $d\geq 2$ be an integer. 

\medskip
\begin{definition}~\label{hooss}
We say that ${\bf c}$ is an oscillating sequence of order $d$ if for every real coefficient polynomial $P$ whose degree is less than or equal to $d$,
\begin{equation}~\label{hoosseq2}
\lim_{N\to \infty} \frac{1}{N} \sum_{n=1}^{N} c_{n}e^{2\pi i P(n)}=0.
\end{equation} 
If (\ref{hoosseq2}) holds for all real coefficient polynomials $P$, then we call ${\bf c}$ a completely oscillating sequence.
\end{definition} 

\medskip
\begin{definition}~\label{hoossas}
We say that ${\bf c}$ is an oscillating sequence of order $d$ in the arithmetic sense if for every real coefficient polynomial $P$ whose degree is less than or equal $d$ and any pair of integers $0\leq l<k$,
\begin{equation}~\label{hossin1}
\lim_{N\to \infty} \frac{1}{N} \sum_{1\leq n\leq N, \;\; n\equiv l \pmod{k}} c_{n}e^{2\pi i P(n)}=0.
\end{equation}
If (\ref{hossin1}) holds for all real coefficient polynomials $P$ and all pairs of integers $0\leq l<k$, then we call ${\bf c}$ a completely oscillating sequence in the arithmetic sense.
 \end{definition}
 
 \medskip
 \begin{proposition}~\label{prop1}
 Definition~\ref{os} and Definition~\ref{osas} are equivalent. And Definition~\ref{hooss} and Definition~\ref{hoossas} are equivalent.  
 \end{proposition}
 
 \begin{proof}
It is clear that Definition~\ref{osas} and Definition~\ref{hoossas} imply Definition~\ref{os} and Definition~\ref{hooss}, respectively,  by taking $l=0$ and $k=1$.

Now, suppose that we have Definition~\ref{hooss} (or Definition~\ref{os}). Then for any real coefficient polynomial 
$Q(x)$, denote 
$$
S_{N}(Q) =\frac{1}{N} \sum_{n=1}^{N} c_{n} e^{2\pi i Q(n)},
$$
we have that 
$$
\lim_{N\to \infty} S_{N}(Q) =0.
$$

For any pair of integers $0\leq l<k$, let 
$$
\omega =e^{2\pi i \frac{1}{k}}.
$$
We use the facts:
$$
 \sum_{j=1}^{k} \omega^{j(n-l)}=k \quad \hbox{if $n-l \equiv 0 \pmod{k}$}
 $$
 and
 $$
 \sum_{j=1}^{k} \omega^{j(n-l)}=0 \quad \hbox{if $n-l \not\equiv 0 \pmod{k}$}.  
 $$
For any real coefficient polynomial $P(x)$, 
$$
\frac{1}{N} \sum_{1\leq n\leq N, n-l \equiv 0 \pmod{k}} c_{n} e^{2\pi i P(n)}=\frac{1}{k}
\frac{1}{N} \sum_{n=1}^{N} c_{n} \Big( \sum_{j=1}^{k} \omega^{j(n-l)}\Big) e^{2\pi i P(n)} 
$$
$$
=\frac{1}{k}\sum_{j=1}^{k} \frac{1}{N}\sum_{n=1}^{N} c_{n} \Big( \omega^{j(n-l)}\Big) e^{2\pi i P(n)}
=\frac{1}{k}\sum_{j=1}^{k}  \Big(\omega^{-jl}\Big) \frac{1}{N}\sum_{n=1}^{N} c_{n} e^{2\pi i (P(n)+\frac{jn}{k})} 
$$
$$
=\frac{1}{k} \sum_{j=1}^{k} \Big( \omega^{-jl}\Big)  S_{N}(Q_{j}),
$$
where 
$$
Q_{j} (x) =P(x) +\frac{xj}{k}, \quad 1\leq j\leq k.
$$
Since 
$$
\lim_{N\to \infty} S_{N}(Q_{j}) =0, \quad \forall\; 1\leq j\leq k,
$$
we have that 
$$
\lim_{N\to \infty} \frac{1}{N} 
\sum_{1\leq n\leq N, n\equiv l \pmod{k}} c_{n} e^{2\pi i P(n)}=
\frac{1}{k} \sum_{j=1}^{k} \Big( \omega^{-jl}\Big) \lim_{N\to \infty} S_{N}(Q_{j}) 
=0.
$$
This completes the proof.
 \end{proof}

\medskip
\begin{remark}~\label{mf}
From Davenport~\cite{Da}, we know that the M\"obius sequence ${\bf u}$ is an oscillating sequence, and from Hua~\cite{Hua}, we know that the M\"obius sequence ${\bf u}$ is a completely oscillating sequence. 
\end{remark}

In~\cite{FJ}, we studied the linear disjointness of oscillating sequences and mean Lyapunov stable (abbreviated MLS) dynamical systems. All MLS dynamical systems have zero topological entropy (see~\cite{LTY}).  In~\cite{JPAMS}, we studied the linear disjointness of oscillating sequences of order $d$ and affine torus maps on the $d$-torus of zero topological entropy for $d\geq 2$ as follows. 

Let $\Z$ and $\R$ mean the sets of integers and real numbers, respectively. Let $d\geq 2$ be an integer. Let $\R^{d}$ be the $d$-Euclidean real vector space and let $\Z^{d}$ be the integer lattice in $\R^{d}$.  Then $\mathbb{T}^{d}=\R^{d}/\Z^{d}$ is the $d$-torus.
 
For ${\bf x}\in \Td$, let ${\bf x}^{t}=(x_{1}, \cdots, x_{d})$ mean the transpose of ${\bf x}$. Let ${\rm GL}(d, \Z)$ be the space 
of all $d\times d-$matrices $A$ of integer entries with determinants $\det(A)=\pm 1$. For any $A\in {\rm GL}(d, \Z)$, the map $A{\bf x}$ is an automorphism of $\Td$. For ${\bf v}\in \Td$ with ${\bf v}^{t} =(v_{1}, \cdots, v_{d})$, we have an affine dynamical system on ${\mathbb T}^{d}$,
\begin{equation}~\label{af}
T_{d, A, {\bf v}} ({\bf x})=A{\bf x}+{\bf v}
\end{equation}
A result of Sinai (see~\cite{Si}) says that $T_{d, A,{\bf v}}$ has zero topological entropy if and only if the absolute values of all eigenvalues of $A$ are $1$; furthermore, a result of Kronecker (see~\cite{K}) implies that in this case, all eigenvalues of $A$ are roots of unity. Therefore, for some positive integer $m$, all the eigenvalues of $A^{m}$ are $1$. Without loss of generality, we only need to consider $A\in {\rm GL}(d,\Z)$ such that all its eigenvalues are $1$. We have 

\medskip
\begin{lemma}~\label{tm}
Suppose $A\in {\rm GL}(d,\Z)$ such that all its eigenvalues are $1$. Then we have a matrix $P\in {\rm GL}(d,\Z)$ whose determinant is $1$ such that
$L=P^{-1}AP\in {\rm GL}(d,\Z)$ is a lower triangular matrix.
\end{lemma} 

See~\cite[Lemma 1]{JPAMS} for a proof. 
Thus, without loss of generality, we only need to consider a lower triangular matrix $A\in {\rm GL}(d,\Z)$ whose entries on the main diagonal are $1$. Under this assumption, the map in (\ref{af}) is an affine skew product,
\begin{equation}~\label{asp}
T_{d, A, {\bf v}} ({\bf x})= \left(
\begin{array}{c}
x_{1}+v_{1}\\
x_{2} +b_{21}x_{1}+v_{2}\\
\vdots\\
x_{d}+b_{d(d-1)} x_{d-1} \cdots+b_{d1}x_{1} +v_{d}  
\end{array}\right)
\end{equation}
where $b_{ij}$, $2\leq i\leq d$, $1\leq j\leq i-1$, are integers.
In~\cite{JPAMS}, I proved the following theorem.

\medskip
\begin{thmx}~\label{thma}
Any oscillating sequence ${\bf c}$ of order $d\geq 2$ and any affine skew product $T_{d, A, {\bf v}}$ in (\ref{asp}) are linearly disjoint. 
\end{thmx}

From the proof of Theorem~\ref{thma}, we can obtain directly that  

\medskip
\begin{corx}~\label{cora}
Any oscillating sequence ${\bf c}$ of order $d\geq 2$ in the arithmetic sense (equivalently, any oscillating sequence ${\bf c}$ of order $d\geq 2$ (see Proposition~\ref{prop1})) and any affine torus map
$T_{d, A,{\bf v}}$ in (\ref{af}) having zero topological entropy are linearly disjoint.
\end{corx}

We are interested in the following problem. 

\medskip
\begin{problem}~\label{gen}
Are any completely oscillating sequence ${\mathbf c}$ (having positive entropy) and any continuous torus map $f: {\mathbb T}^{d}\to {\mathbb T}^{d}$ having zero topological entropy linearly disjoint?
\end{problem}


\medskip
\medskip
\medskip

\section{A Generalized Notation of Linear Disjointness}~\label{gn}

To demonstrate that all orders of oscillation needed in the study of Problem~\ref{gen}, 
we consider a torus continuous map $f: {\mathbb T}^{d}\to {\mathbb T}^{d}$ having zero topological entropy whose lift $\widetilde{T}$ has the form 
$$
\widetilde{T} ({\bf x})= A{\bf x} +C({\bf x}): {\mathbb R}^{d}\to {\mathbb R}^{d},
$$
where $A\in {\rm GL}(d, \Z)$ with all absolute values of eigenvalues $1$ and $C({\bf x})$ is a continuous periodic map, that is, $C({\bf x}+{\bf 1})=C({\bf x})$ and ${\bf 1}^{t} =(1, \cdots, 1)$. If we consider a polynomial approximation
$P({\bf x})$ of $C({\bf x})$, then we have a continuous map
$$
F ({\bf x})= A{\bf x} +P({\bf x}): {\mathbb R}^{d}\to {\mathbb R}^{d}.
$$
However, $F({\bf x})$ may not induce a continuous torus map.
So, we would like to generalize the definition of linear disjointness.

Let $Y$ be a Polish space, not necessarily compact. 
Suppose $\sim$ is an equivalence relation on $Y$ such that the quotient space $X=Y/\sim$ is a compact metric space.
Let $\pi: Y\to X$ be the natural projection, that is, $\pi (y)=[y]$, where $[y]$ means the equivalence class of $y$.  

For any continuous map $f: Y\to Y$, we have a (discrete) continuous dynamical system $\{f^{n}\}_{n\in \{0\}\cup \N}$ on $Y$. We just use $f$ to mean this dynamical system. 
Let $C(X)$ be the space of all continuous functions on $X$ with the maximum norm $\|\phi\|=\max_{x\in X} |\phi (x)|$. We can have a generalized notation of linear disjointness. 

\medskip 
\begin{definition}~\label{gld}
We say that ${\bf c}$ and a continuous dynamical system $f:Y\to Y$ are (asymptotically) linearly disjoint if for any $\phi\in C(X)$,
\begin{equation}\label{disjointness}
  \lim_{N\to \infty} \frac{1}{N}\sum_{n=1}^N c_n \phi (\pi (f^{n}y)) =0,\quad \forall\; y\in Y.
\end{equation}
\end{definition}

We have a generalized conjecture (refer to Conjecture~\ref{sc}).  

\medskip
\begin{conjecture}~\label{gsc}
The M\"obius sequence ${\bf u}$ and any continuous dynamical system $f: Y\to Y$ having zero entropy are linearly disjoint.
\end{conjecture}

\medskip

\begin{remark}~\label{pl}
In the above conjecture, we use the definition of entropy given in~\cite{DJ} for a continuous dynamical system $f$ from a Polish space $Y$ into itself. When the space $Y=X$ itself is compact, the entropy defined in~\cite{DJ} is the same as the topological entropy. 
\end{remark}

\section{Polynomial skew products.}~\label{psp}

Consider $Y={\mathbb R}^{d}$, $X={\mathbb T}^{d}$.  Then $\pi: {\mathbb R}^{d}\to {\mathbb T}^{d}$ is the universal cover. 
We call a map $f$ a polynomial skew product of degree $k$ if 
\begin{equation}~\label{psp2}
f ({\bf x})= \left(
\begin{array}{c}
x_{1}+ a\\
x_{2} +h_{2} (x_{1})\\
x_{3} +h_{3}(x_{1}, x_{2})\\
\vdots\\
x_{d}+h_{d} (x_{1}, x_{2}, \cdots, x_{d-1})
\end{array}\right)
\end{equation}
where $a$ is a constant and $h_{2} (x_{1})$ is a polynomial of $x_{1}$ of degree $\leq k$, 
$h_{3}(x_{1}, x_{2})$ is a polynomial of $x_{1}$ and $x_{2}$ of degree $\leq k$, $\cdots$, 
$h_{d}(x_{1}, \cdots, x_{d-1})$ is a polynomial of $x_{1}, \cdots, x_{d-1}$ of degree $\leq k$. 
We call $f$ a simple polynomial skew product of degree $k$ if it is in the form
\begin{equation}~\label{psp1}
f ({\bf x})= \left(
\begin{array}{c}
x_{1}+ a\\
x_{2} +h_{2} (x_{1})\\
x_{3} +b_{32}x_{2} +h_{3}(x_{1})\\
\vdots\\
x_{d}+b_{d(d-1)}x_{d-1} \cdots +b_{d2}x_{2}+h_{d} (x_{1})
\end{array}\right)
\end{equation}
where $a$ is a constant and $b_{ij}$, $3\leq i\leq d$, $2\leq j\leq i-1$, are integers and $h_{i} (x_{1})$, $2\leq i\leq d$, are polynomials of $x_{1}$ of degree $\leq k$.
One of the results in this paper, which generalizes Theorem~\ref{thma}, is the following theorem.

\medskip
\begin{theorem}~\label{main}
Suppose $d\geq 2$ and $k\geq 1$. Any oscillating sequence ${\bf c}$ of order $m=d+k-1$ and any simple polynomial skew products $f ({\bf x})$ of degree $k$ in (\ref{psp1}) are linearly disjoint.
\end{theorem}

\medskip
\begin{corollary}~\label{mainc}
Suppose $A\in {\rm GL}(d,\Z)$ with all absolute values of the eigenvalues $1$. Let $q\geq 1$ be the smallest integer such that all the eigenvalues of $A^{q}$ are $1$. Suppose 
$$
f^{q}({\bf x}) =A^{q}{\bf x} +{\bf h} (x_{1})
$$
where ${\bf h}^{t} (x_{1}) =(a, h_{2} (x_{1}), \cdots, h_{d}(x_{1}))$ and $a$ is a constant and $h_{i}(x_{1})$, $2\leq i\leq d$, are polynomials of $x_{1}$ of degree $\leq k$ for $k\geq 1$. 
Then, any oscillating sequence ${\bf c}$ of order $m=d+k-1$ and $f$ are linearly disjoint. 
\end{corollary}

\begin{proof}[Proof of Theorem~\ref{main}.]
Let ${\bf k}\in \Z^{d}$ with ${\bf k}^{t}=(k_{1}, \cdots,k_{d})$, define
$$
e({\bf k}\cdot {\bf x}) = e^{2\pi i (k_{1}x_{1}+\cdots + k_{d}x_{d})}.
$$
From the Stone-Weierstrass theorem, 
the set 
$S=\Big\{ e({\bf k}\cdot {\bf x})\Big\}_{{\bf k}\in \Z^{d}}$ 
forms a dense subset in $C(\Td)$.
A trigonometric polynomial $p$ is a linear combination of elements in $S$. We can write $p$ as 
$$
p({\bf x}) = \sum_{m_{1}\leq k_{1}\leq s_{1}}\cdots \sum_{m_{d}\leq k_{d}\leq s_{d}} a_{\bf k} e^{2\pi i (k_{1}x_{1}+\cdots +k_{d}x_{d})}.
$$
For any $\phi\in C(\Td)$, we have a sequence of trigonometric polynomials 
\begin{equation}~\label{stp}
p_{q}({\bf x}) = \sum_{m_{1q}\leq k_{1}\leq s_{1q}}\cdots \sum_{m_{dq}\leq k_{d}\leq s_{dq}} a_{{\bf k}, q} e^{2\pi i (k_{1}x_{1}+\cdots +k_{d}x_{d})}.
\end{equation}
such that $\|\phi-p_{q}\|\to 0$ as $q\to \infty$. The sequence $\{p_{q}\}_{q\in\N}$ is called the trigonometric approximation of $\phi$.
Consider
$$
S_{N}\phi ({\bf x}) = \frac{1}{N} \sum_{n=1}^{N} c_{n} \phi(f^{n} {\bf x}).
$$
 
For any $\epsilon>0$, we have an integer $r>0$ such that
$$
\| \phi-p_{r}\| <\frac{\epsilon}{2C^{\frac{1}{\lambda}}}.
$$
Then 
$$
S_{N}\phi ({\bf x}) =\Big( \frac{1}{N} \sum_{n=1}^{N} c_{n} \big( \phi(f^{n} {\bf x}) -p_{r} (f^{n} {\bf x})\big)\Big) +\Big( \frac{1}{N} \sum_{n=1}^{N} c_{n} p_{r} (f^{n} {\bf x})\Big) =I+II.
$$

For the estimation of $I$, we apply the H\"older inequality. Let $\lambda$ and $C$ be the numbers in (\ref{cn}) and $\lambda'>1$ be the dual number of $\lambda$, that is, 
$$
\frac{1}{\lambda} +\frac{1}{\lambda'}=1.
$$ 
Then we have that
$$
|I|\leq \Big( \frac{1}{N} \sum_{l=1}^{N} |c_{n}|^{\lambda}\Big)^{\frac{1}{\lambda}} 
 \Big( \frac{1}{N} \sum_{n=1}^{N} |\phi(f^{n} {\bf x})-p_{r}(f^{n} {\bf x})|^{\lambda'}\Big)^{\frac{1}{\lambda'}}
\leq C^{\frac{1}{\lambda}} \cdot \frac{\epsilon}{2C^{\frac{1}{\lambda}}} =\frac{\epsilon}{2}.
$$

For the estimation of $II$, let ${\bf x}_{n}=f^{n}{\bf x}$ for $n=1, 2,\cdots$.
Denote 
$$
{\bf x}_{n}^{t}=(x_{1}^{n},\cdots, x_{d}^{n}).
$$
Then 
$$
x_{1}^{n} = x_{1} + an
$$
is a polynomial of $n$ of degree less than or equal to $1$. Assume $x_{1}^{0}=x_{1}$. 

Suppose $h_{2} (x_{1}) = \sum_{i=0}^{k} r_{i} \cdot (x_{1})^{i}$. Then
$$
x_{2}^{n} = x_{2} +\sum_{l=0}^{n-1} h_{2}(x_{1}^{l})= x_{2}+\sum_{l=0}^{n-1} \sum_{i=0}^{k} r_{i} \cdot (x_{1}+la)^{i}
$$
$$
= \sum_{l=0}^{n-1} \sum_{i=0}^{k} s_{i} l^{i}
=\sum_{i=0}^{k} s_{i} \sum_{l=0}^{n-1} l^{i}= \sum_{i=0}^{k} s_{i} Q_{i}(n)
$$
where all $s_{i}$ are numbers only depending on $r_{i}$, $k$, $x_{1}$, and $a$ but not on $n$, and
$$
Q_{i}(n) = \sum_{l=0}^{n-1} l^{i}
$$ 
are polynomials of $n$ of degrees $i+1$ for $0\leq i\leq k$. 
Thus, $x_{2}^{n}$ is a polynomial of $n$ of degree less than or equal to $k+1$. Assume $x_{2}^{0}=x_{2}$.

Now
$$
x_{3}^{n} = x_{3} +b_{32} \Big( \sum_{l=0}^{n-1} x_{2}^{l} \Big) +\sum_{l=0}^{n-1} h_{3} (x_{1}^{l}).
$$
Similar to the calculation of $x_{2}^{n}$, we have that $\sum_{l=0}^{n-1} h_{3} (x_{1}^{l})$ is a polynomial of $n$ of degree less than or equal to $k+1$. Since $x_{2}^{n}$ is a polynomial of $n$ of degree less than or equal to $k+1$, $\sum_{l=0}^{n-1} x_{2}^{l}$ is a polynomial of $n$ of degree less than or equal to $k+2$. Thus, $x_{3}^{n}$ is a polynomial of $n$ of degree less than or equal to $k+2$. 
Inductively, we have that $x_{i}^{n}$ is a polynomial of $n$ of degree of $k+i-1$ for $4\leq i\leq d$. Furthermore,
$$
P_{{\bf k}} (n) = k_{1}x_{1}^{n}+\cdots +k_{d}x_{d}^{n}
$$
is a polynomial of $n$ of degree less than or equal to $d+k-1$ and 
$$
p_{r} (f^{n} {\bf x}) =  
\sum_{m_{1r} \leq k_{1}\leq s_{1r}} \cdots \sum_{ m_{dr} \leq k_{d}\leq s_{dr}} a_{{\bf k},r} e^{2\pi i P_{{\bf k}} (n)}.
$$

We can now continue to estimate 
$$
|II| =\Big|\frac{1}{N} \sum_{n=1}^{N} c_{n} \sum_{m_{1r} \leq k_{1}\leq s_{1r}} \cdots \sum_{m_{dr}\leq k_{d}\leq s_{dr}} a_{{\bf k},r} e^{2\pi i P_{{\bf k}} (n)}\Big| 
$$
$$
= \Big |\sum_{m_{1r}\leq k_{1}\leq s_{1r}}\cdots \sum_{m_{dr}\leq k_{d}\leq s_{dr}} a_{{\bf k},r} \frac{1}{N} \sum_{n=1}^{N} c_{n} e^{2\pi i P_{{\bf k}} (n)} \Big|.
$$
Let 
$$
L=\max \{ |m_{1r}|, \cdots, |m_{dr}|, |s_{1r}|, \cdots, |s_{dr}|, |a_{{\bf k},r}| \; ; \; m_{jr} \leq k_{j}\leq s_{jr}, 1\leq j\leq d\}.
$$
Since ${\bf c}$ is an oscillating sequence of order $d+k-1$, we can find an integer $M>r$ such that for $N>M$,
$$
\Big|\frac{1}{N} \sum_{n=1}^{N} c_{n} e^{2\pi i P_{{\bf k}} (n)} \Big|< \frac{\epsilon}{2L^{d}}, \quad \forall\; m_{1r} \leq k_{1}\leq s_{1r}, \cdots, m_{dr}\leq k_{n}\leq s_{dr}.
$$
This implies that 
$$
|II|< \epsilon/2.
$$

Combining the estimations of $I$ and $II$, we get that, for all $N>M$,
$$
|S_{N}\phi ({\bf x})|< \epsilon.
$$
This says that $\lim_{N\to \infty} S_{N}\phi({\bf x}) =0$. 
This completes the proof.
\end{proof}

\begin{proof}[Proof of Corollary~\ref{mainc}]
Using the same proof as that of Theorem~\ref{main}, we can show that any oscillating sequence ${\bf c}$ of order $d\geq 2$ in the arithmetic sense and $f$ are linearly disjoint. Then, using Proposition~\ref{prop1}, we get any oscillating sequence ${\bf c}$ of order $d\geq 2$ and $f$ are linearly disjoint. 
\end{proof}

Theorem~\ref{main} can be generalized to the following theorem.
 
\medskip
\begin{theorem}~\label{main2}
Suppose $f$ is a polynomial skew product of degree $k$ in the form (\ref{psp2}). 
Then there is a positive integer $m=m(d, k, h_{1}, \cdots, h_{d})$ such that any oscillating sequence ${\bf c}$ of order $m$ and $f$ are linearly disjoint. 
\end{theorem}

\begin{proof}
Most of the proof is similar to the proof of Theorem~\ref{main}. Here, we outline some different steps. Let ${\bf x}_{n}=f^{n}{\bf x}$ for $n=1, 2, \cdots$ and ${\bf x}_{0}$.
Denote ${\bf x}_{n}^{t}=(x_{1}^{n},\cdots, x_{d}^{n})$. From the proof of Theorem~\ref{main}, we know that $x_{1}^{n}$ and $x_{2}^{n}$ are polynomials of $n$ of degrees less than or equal to $k+2$. 

Let $h_{3} (x_{1}, x_{2}) = \sum_{0\leq i+j\leq k} r_{ij} (x_{1})^{i}(x_{2})^{j}$. Then 
$$
x_{3}^{n} = x_{3} +\sum_{l=0}^{n-1} h_{3} (x_{1}^{l}, x_{2}^{l}),
=x_{3} +\sum_{l=0}^{n-1} \sum_{0\leq i+j\leq k} r_{ij} \cdot (x_{1}^{l})^{i} \cdot (x_{2}^{l})^{j}.
$$

$$
=x_{3} +\sum_{0\leq i+j\leq k}r_{ij} \cdot \Big(\sum_{l=0}^{n-1}  (x_{1}^{l})^{i} (x_{2}^{l})^{j}\Big).
$$
Here $\sum_{l=0}^{n-1}  (x_{1}^{l})^{i} (x_{2}^{l})^{j}$ is a polynomial of $n$. 
Inductively, we have that  
$$
x_{i}^{n} = x_{i} + \sum_{l=0}^{n-1} h_{i}(x_{1}^{l}, \cdots, x_{i-1}^{l})
$$
is a polynomial of $n$ for $4\leq i\leq d$. 

Let $m$ be the maximum of all degrees of $x_{i}^{n}$ for $1\leq i\leq d$. 
The rest of the proof is the same as in the proof of Theorem~\ref{main}

\end{proof}

\medskip
\begin{corollary}~\label{maincg}
Suppose $A\in \hbox{GL}(d,\Z)$ with all absolute values of the eigenvalues $1$. Let $q\geq 1$ be the smallest integer such that all the eigenvalues of $A^{q}$ are $1$. Suppose 
$$
f^{q}({\bf x}) =A^{q}{\bf x} +{\bf h} 
$$
where 
${\bf h}^{t} =(a, h_{2} (x_{1}), h_{3}(x_{1}, x_{2}), \cdots, h_{d}(x_{1}, \cdots, x_{d-1}))$ and $h_{i}(x_{1}, \cdots, x_{i-1})$, $2\leq i\leq d$, are polynomials of $x_{1}, \cdots, x_{i-1}$ degree $\leq k$. 
Then there is a positive integer $m=m(d, k, h_{1}, \cdots, h_{d})$ such that any oscillating sequence ${\bf c}$ of order $m$ and $f$ are linearly disjoint. 
\end{corollary}

 \section{Minimal Mean Attractable and Minimal Quasi-Discrete Spectrum Dynamical Systems.}~\label{mmamqds}  
 
 In~\cite{FJ}, we extended the concept of MLS dynamical systems of Fomin~\cite{Fo} to the concept of minimal MLS (abbreviated MMLS) dynamical systems. We also define a new concept called minimal mean attractable (abbreviated MMA) dynamical systems in~\cite{FJ}. We studied the linear disjointness of any oscillating sequence and any MMA and MMLS dynamical system by using these two concepts. In this paper, I will extend the concept of the quasi-discrete spectrum (abbreviated QDS) dynamical systems of Abramov, Hahn, and Parry~\cite{Ab,HP} to the concept of minimal QDS (abbreviated MQDS) dynamical systems of order $d$ and study the linear disjointness of any oscillating sequence of order $d$ and any MMA and MQDS (d) dynamical system.   

 \medskip
Suppose $X$ is a compact metric space. Let $f: X\to X$ be a continuous dynamical system. A subset $K$ of $X$ is called minimal (respective to $f$) if for any $x\in K$, the closure of the forward orbit of $x$ is $K$, that is, 
$$
\overline{\{ f^{n}x\}_{n=0}^{\infty}}=K.
$$

\medskip  
\begin{definition}[MMA]~\label{mma}
Suppose $K\subseteq X$ is minimal. We say $x\in X$ is {\em mean attracted} to $K$ if for any $\epsilon >0$, there is a point $z=z_{\epsilon, x}\in K$ such that
\begin{equation}\label{MA}
  \limsup_{N\to\infty} \frac{1}{N} \sum_{n=1}^{N} d(f^n x, f^n z) <\epsilon.
\end{equation}
The {\em basin of attraction} of $K$, denoted as ${\rm Basin}(K)$, is defined to be the set of all points $x$ that are mean attracted to $K$. 
We call $f$ {\em minimal mean attractable} (abbreviated MMA) if
\begin{equation}\label{Decomp}
  X= \bigcup_{K} \mbox{\rm Basin}(K)
\end{equation}
where $K$ varies among all minimal subsets of $X$.
\end{definition}

\medskip
\begin{remark}
It is clear that
$K \subseteq \hbox{\rm Basin}(K)$. The union in the above definition could be uncountable.
\end{remark}

We use~\cite{HM} as a general reference in reviewing a QDS dynamical system. Suppose $G$ is an Abelian group in multiplication and $\Lambda: G\to G$ is a homomorphism. Consider another homomorphism $\Phi: G\to G$ defined as 
$$
\Phi g :=g\cdot \Lambda g, \quad g\in G.
$$
Then we have that
$$
\Phi^{n} g = \prod_{j=0}^{n} (\Lambda^{j} g)^{\binom{n}{j}}, \quad g\in G.
$$
Let $G_{n} =\ker (\Lambda^{n})$ be a subgroup of $G$. Then, we have a filter 
$$
\{ 1\}=G_{0}\subseteq G_{1}\subseteq \cdots \subseteq G_{n} \subseteq G_{n+1}\subseteq \cdots \subseteq G
$$
and $\Lambda: G_{n}\to G_{n-1}$ for all $n\geq 1$. Note that $\Phi: G_{n}\to G_{n}$ is an automorphism for every $n\geq 1$ (one can check it by using induction).
The homomorphism $\Lambda: G\to G$ is called nilpotent if $G=G_{n}$ for some $n\geq 0$. It is called quasi-nilpotent if $G=\cup_{n=0}^{\infty} G_{n}$.

A triple $(G, \Lambda, \iota)$ is called a signature if $G$ is an Abelian group, $\Lambda: G\to G$ is a quasi-nilpotent homomorphism and 
$$
\iota : G_{1}\to \TT
$$
is an injective homomorphism (i.e., monomorphism), where $G_{1}=\ker(\Lambda)$. The order of the signature $(G, \Lambda, \iota)$ is 
$$
{\rm ord} (G, \Lambda, \iota) := \inf \{ n\in \N\;|\; G=G_{n}\}\in \N\cup \{\infty\}.
$$ 
Note that the order is infinite if $G\not= G_{n}$ for all $n\in \N$, that is, $G$ is not nilpotent.

Now let us return to $f: X\to X$. Suppose $K\subseteq X$ is minimal. Consider the Koopman operator $\Phi: C (K) \to C (K)$ defined as 
$$
\Phi \phi = \phi\circ f, \quad \phi \in C(K). 
$$
Let
$$
G= \{ g\in C(K)\;|\; |g|=1\}.
$$ 
Then, it is an Abelian group under multiplication. Consider the homomorphism $\Lambda: G\to G$ defined as
$$
\Lambda g = \Phi g\cdot \overline{g}, \quad g\in G. 
$$
Then 
$$
\Phi g= g\cdot \Lambda g, \quad g\in G.
$$
One can see that $G_{1} = Fix (\Phi) \cap G$ where $Fix(\Phi)$ is the set of all fixed points of $\Phi$. 
Since $K$ is minimal, we have that $G_{1}={\mathbb T}$ and 
$$
Fix(\Phi) ={\rm span}_{\C} (G_{1})=\C.
$$ 
Let us consider $G_{2}=\ker (\Lambda^{2})$. For any $g\in G_{2}$, we have $\Lambda^{2}g=1$. This implies that 
$$
\Phi^{2} g = g \cdot (\Lambda g)^{2} \cdot\Lambda^{2} g= \Lambda g\cdot \Phi g
$$
Since $\Lambda g\in G_{1}$, $\lambda =\Lambda g$ with $|\lambda|=1$ is a unimodular eigenvalue of $\Phi$ with an unimodular eigenvector $g'=\Phi g\in G_{2}$. Conversely, suppose $\lambda$ is a unimodular eigenvalue of $\Phi$ with a unimodular eigenvector $g$, that is, $\Phi g=\lambda g$ and $|\lambda|=1$. Then $|g|=1 \in G_{1}$ and $\Lambda g =\Phi g\cdot \overline{g}=\lambda g\cdot \overline{g}=\lambda \in G_{1}$. This says $g\in G_{2}$. Thus, $G_{2}$ is the set of all unimodular eigenvectors of $\Phi$. In general, elements in $G_{n}$ are called unimodular quasi-eigenvectors of order $n-1$. We say that $f: K\to K$ is quasi-discrete spectrum (abbreviated QDS) if the linear hull of all unimodular quasi-eigenvectors is dense in $C(K)$, that is, $G=\cup_{n=0}^{\infty} G_{n}$ and 
$$
\overline{{\rm span}_{\C} (G)} = C(K).
$$
We say that $f: K\to K$ is quasi-discrete spectrum of order $d$ (abbreviated QDS(d)) if it is QDS and ${\rm ord} (G, \Lambda, \iota)=d+1$. 

\medskip
\begin{definition}[MMA and MQDS]~\label{mmamqdsd}
We call $f: X\to X$ MMA and MQDS (d) if $f$ is MMA and for every minimal set $K\subseteq X$, $f: K\to K$ is QDS (k)
for some $k\leq d$.
\end{definition}

\medskip
\begin{theorem}~\label{main3}
Any oscillating sequence ${\bf c}$ of order $d\geq 1$ and any MMA and MQDS (d) $f$ are linearly disjoint.
\end{theorem}
 
\begin{proof}
Let $\lambda > 1$ and $C$ be the numbers in (\ref{cn}). Let $\lambda'> 1$ be the dual number of $\lambda$, that is, 
$$
\frac{1}{\lambda} + \frac{1}{\lambda'}=1.
$$

For any $x\in X$ and any $\phi\in C(X)$, we need to prove that (\ref{ld}) holds. 
Since $f$ is MMA, $x\in Basin (K)$ for some minimal subset $K\subseteq X$. 
Assume that we already know that (\ref{ld}) holds for $z\in K$. For an arbitrarily small number $\epsilon >0$, 
the uniform continuity of $\phi$ implies that there is a $\delta >0$ such that 
$$
|\phi(u)- \phi(v)| <\frac{\epsilon}{2^{\frac{1+\lambda'}{\lambda'}} C^{\frac{1}{\lambda}}}
$$ 
whenever $d(u, v)<\delta$.
Since $f$ is MMA, there exists $z=z_{\delta,x}\in K$
such that
  \begin{equation*}\label{MA2}
  \limsup_{N\to\infty} \frac{1}{N} \sum_{n=1}^{N} d(f^n x, f^n z) <\delta^2.
\end{equation*}
Let $E=\{n\ge 1\;|\; d(f^n x, f^n z)\ge \delta\}$. Then the above inequality implies an estimation of the upper density $\overline{D}(E)$ of $E$ in $\N$:
$$
\overline{D}(E)=\limsup_{n\to \infty} \frac{\#(E\cap [1,N])}{N} \le \delta
$$
because
$$
 \delta \sharp (E\cap [1, N]) \le \sum_{n=1}^{N} d(f^n x, f^n z).
$$
Consider 
$$
S_N\phi(x)=\frac{1}{N} \sum_{n=1}^{N} c_{n} \phi (f^{n}x).
$$
Write
$$
   S_N\phi(x) = \Big( S_N\phi(x) - S_N\phi(z)\Big) + S_N\phi(z) =I+II
$$
Since we assume that (\ref{ld}) holds for $z\in K$, we have a $N_{0}>0$ such that 
$$
|II|=|S_N\phi(z)|<\epsilon/2
$$ 
for all $N>N_{0}$. From the H\"{o}lder inequality,  
$$
|I|= |S_N\phi (x) - S_N \phi (z)| \leq \Big( \frac{1}{N} \sum_{n=1}^{N} |c_{n}|^{\lambda} \Big)^{\frac{1}{\lambda}} \Big(\frac{1}{N} \sum_{n=1}^{N} |\phi(f^n x) - \phi(f^n z)|^{\lambda'}\Big)^{1/\lambda'}
$$
$$
\leq C^{\frac{1}{\lambda}} \Big(\frac{1}{N} \sum_{n=1}^{N} |\phi(f^n x) - \phi(f^n z)|^{\lambda'}\Big)^{1/\lambda'}.
 $$
This implies that 
$$
|I|\leq C^{\frac{1}{\lambda}} \Big( \frac{1}{N} \sum_{n\in [1,N]\setminus E} |\phi(f^n x) - \phi(f^n z)|^{\lambda'} +\frac{1}{N} \sum_{n\in E} |\phi(f^n x) - \phi(f^n z)|^{\lambda'}\Big)^{1/\lambda'} 
$$
$$
\leq C^{\frac{1}{\lambda}} \Big( \frac{\epsilon^{\lambda'}}{2^{1+\lambda'} C^{\frac{\lambda'}{\lambda}}} +(2\|\phi\|_{\infty})^{\lambda'}\delta\Big)^{\frac{1}{\lambda'}}.
$$
Here we can take $\delta$ small enough such that 
$$
(2\|\phi\|_{\infty})^{\lambda'}\delta < \frac{\epsilon^{\lambda'}}{2^{1+\lambda'} C^{\frac{\lambda'}{\lambda}}}.
$$ 
So we have that 
$$
|S_N\phi (x)|\leq |I|+|II|\leq \epsilon
$$
for all $N>N_{0}$. We proved (\ref{ld}) under the assumption.

Now we prove the assumption that for any $z\in K$, (\ref{ld}) holds. Since $f: K\to K$ is a QDS (d) flow, we have that 
$$
C(K)=\overline{{\rm span}_{\C} (G_{d+1})}.
$$ 
Similar to the proof of Theorem~\ref{main}, we need only to prove (\ref{ld}) for any $g\in G_{d+1}$ and any $z\in K$.
Since $\Lambda^{d+1}g=1$, For any $n\geq d+1$, we have 
$$
\Phi^{n} g(z) =g\circ f^{n}(z) = \prod_{j=0}^{n} (\Lambda^{j}g (z))^{\binom{n}{j}} = \prod_{j=0}^{d} (\Lambda^{j}g (z))^{\binom{n}{j}}.  
$$
Since $\Lambda^{j} g \in G_{d+1-j}$ for $0\leq j \leq d$, we have $\Lambda^{j} g (z)= e^{2\pi i \theta_{j}}$ for some real number $\theta_{j}\in [0, 2\pi]$. This implies that 
$$
\Phi^{n} g(z) = e^{ 2\pi i \sum_{j=0}^{d} \theta_{j} \binom{n}{j}} =e^{2\pi i P (n)}
$$
where 
$$
P (n) =\sum_{j=0}^{d} a_{j} n^{j}
$$ 
is a real-coefficient polynomial of $n$ whose degree is less than or equal to $d$.
Now 
$$
S_{N} g (z) = \frac{1}{N} \sum_{n=1}^{N} c_{n} g\circ f^{n} =  \frac{1}{N} \sum_{n=1}^{d} c_{n} \Phi^{n}g (z) + \frac{1}{N} \sum_{n=d+1}^{N} c_{n} \Phi^{n} g(z) 
$$
$$
= \frac{1}{N} \sum_{n=1}^{d} c_{n} \Phi^{n}g (z) + \frac{1}{N} \sum_{n=d+1}^{N} c_{n} e^{2\pi i P(n)}.
$$
Since ${\bf c}$ is an oscillating sequence of order $d$,
$$
\lim_{n\to \infty} \frac{1}{N} \sum_{n=d+1}^{N} c_{n} e^{2\pi P(n)}=0.
$$
Therefore, we have that 
$$
\lim_{N\to \infty} S_{N} g (z) =0.
$$
This completes the proof of the assumption and thus, the proof of the theorem.
\end{proof}

\medskip
\begin{remark}
From Theorem~\ref{main}, Theorem~\ref{main2}, and Theorem~\ref{main3}, we can see that a general nonlinear continuous torus map on ${\mathbb T}^{d}$ having zero topological entropy will not be MQDS(d). 
\end{remark}  
  
\section{Multi-Linearly Disjoint Sequences.}~\label{cssec}  

It has been asked whether it is enough to study Sarnak's conjecture just by studying the orders of oscillation sequences. We can construct examples of completely oscillating sequences from automatic sequences such as the Thue-Morse and Rudin-Shapiro sequences and many other automatic sequences not correlated with periodic sequences (see~\cite{A1,BKM}). These examples have zero topological entropy, as they are viewed as dynamical systems. These sequences and themself viewed as dynamical systems are not linearly disjoint. Thus, they are not our main concern in the study of linear disjointness. We are more interested in completely oscillating sequences having positive entropy, such as the M\"obius sequence ${\bf u}$. In~\cite{AJ}, we have constructed completely oscillating sequences other than the M\"obius sequence (but still have positive entropy (refer to Remark~\ref{pl})) as follows. 

Let $\mathcal{C}^{k}_{+}((1,\infty))$ be the space of all positive real-valued $k$-time continuously differentiable
functions on $(1,\infty)$, whose $i$-th derivative is nonnegative for $i\le k$. We proved the following theorem in~\cite{AJ}.  

\medskip
\begin{thmx}~\label{ajthm}
Take $g\in \mathcal{C}^{2}_{+}((1,\infty))$. Then, for a fixed real
number $\alpha\not=0$ and almost all real numbers $\beta>1$ 
(alternatively, for a fixed real number $\beta>1$ and almost all real numbers $\alpha$),
$$
{\bf c} =\big( e^{2\pi i (\alpha \beta^{n}g(\beta))}\big)_{n\in \N}
$$
are completely oscillating sequences.
\end{thmx}

\medskip
We would like to know that 

\medskip
\begin{problem}  
Are any completely oscillating sequence ${\bf c}$ in Theorem~\ref{ajthm} and any dynamical system of zero topological entropy $f: X\to X$ linearly disjoint?
\end{problem}

Given the study of the above problem, we would like to study those multi-linearly disjoint sequences ${\bf c}$.
 
\medskip
\begin{definition}~\label{cs}
We call ${\bf c}$ multi-linearly disjoint if 
\begin{equation}~\label{cse}
\lim_{N\to \infty} \frac{1}{N} \sum_{n=1}^{N} c_{n+l_{1}}^{k_{1}} c_{n+l_{2}}^{k_{2}} \cdots c_{n+l_{r}}^{k_{r}} =0
\end{equation} 
for all $r\geq 1$, all choices of integers $0\leq l_{1} <l_{2}< \cdots <l_{r}$ and $k_{j}\in \N$, $j=1, 2, \cdots, r$, such that not all $c_{n+l_{j}}^{k_{j}} =|c_{n+l_{j}}|^{k_{j}}$ hold.
\end{definition}

\medskip
This definition is motivated by the Chowla conjecture in~\cite{C}.

\medskip
\begin{conjecture}
For any $r$ integers $0\leq l_{1}<l_{2}<\cdots <l_{r}$ and $r$ indexes $i_{1},i_{2}\cdots, i_{r}$ with $i_{k}=1$ or $2$, $1\leq k\leq r$,  but not all $2$, 
$$
\lim_{N\to \infty} \frac{1}{N}\sum_{n=1}^{N} \mu (n+l_{1})^{i_{1}} \mu(n+l_{2})^{i_{2}}\cdots \mu (n+l_{r})^{i_{r}}=0.
$$
\end{conjecture}

Using the same proof as that of Theorem~\ref{ajthm}, we have 

\medskip
\begin{theorem}~\label{mlds}
Take $g\in \mathcal{C}^{2}_{+}((1,\infty))$. Then, for a fixed real
number $\alpha\not=0$ and almost all real numbers $\beta>1$ 
(alternatively, for a fixed real number $\beta>1$ and almost all real numbers $\alpha$),
$$
{\bf c} =\big( e^{2\pi i (\alpha \beta^{n}g(\beta))}\big)_{n\in \N}
$$
are multi-linearly disjoint sequences.
\end{theorem}

To have a self-contained paper, we provide a detailed proof of Theorem~\ref{mlds}. 
Consider 
$$
c_{n+l_{1}}^{k_{1}} c_{n+l_{2}}^{k_{2}} \cdots c_{n+l_{r}}^{k_{r}} = e^{2\pi i \big( \sum_{i=1}^{r} k_{i} \big( \alpha \beta^{n+l_{i}} g(\beta)\big)\big)}.
$$
This says that (\ref{cse}) relates with the uniformly distributed modulo $1$ (abbreviated u.d.$1$) of the sequence 
$$
\Big( \sum_{i=1}^{r} k_{i} \big(\alpha \beta^{n+l_{i}} g(\beta)\big) \Big)_{n\in \N}
$$  
on the unit interval $[0,1]$. For a real number $x$, let $\{x\}=x \mod{1}$ be the fractional part of $x$.
 
\medskip
\begin{definition} 
We say that a sequence ${\bf x}=(x_{n})_{n\in \N}$ of real numbers is u.d.$1$ if for any $0\leq a<b\leq 1$,
we have
$$
\lim_{N\to \infty} \frac{\#(\{ n\in [1,N] \;|\; \{x_{n}\} \in [a,b]\})}{N}=b-a.
$$
 \end{definition}

\medskip
The following Weyl criterion connects (\ref{cse}) and u.d.$1$.

\medskip
\begin{thmx}~\label{wc}
The sequence ${\bf x}=(x_{n})_{n\in \N}$ is u.d.$1$ if and only if 
$$
\lim_{N\to \infty} \frac{1}{N} \sum_{n=1}^{N} e^{2\pi i h x_{n}} =0 \quad \hbox{for all integers $h\not=0$}.
$$
\end{thmx}

The following theorem is essentially~\cite[Theorem 1]{AJ} (that is, take all $l_{i}=0$, $1\leq i\leq r$, in the proof in~\cite{AJ}). 

\medskip
\begin{theorem}~\label{thmuni}
Let us take $g\in \mathcal{C}^{2}_{+}((1,\infty))$. Then, for a fixed real
number $\alpha\not=0$ and almost all real numbers $\beta>1$ 
(alternatively, for a fixed real number $\beta>1$ and almost all real numbers $\alpha$) and for all $r\geq 1$, all choices of integers $0\leq l_{1} <l_{2}< \cdots <l_{r}$ and $k_{i}\in \N$, $i=1, 2, \cdots, r$, sequences 
\begin{equation}~\label{main1eq}
\Big( \sum_{i=1}^{r} k_{i} \big( \alpha \beta^{n+l_{i}}g(\beta)\big) 
\Big)_{n\in \N}
\end{equation}
are u.d.$1$.  
\end{theorem}

Thus, Theorem~\ref{mlds} is just a corollary of Theorem~\ref{thmuni} and Teorem~\ref{wc}. 
The following theorem is credited as Koksma's Theorem in~\cite{KN}, which we will use in the proof of Theorem~\ref{thmuni}.

\medskip
\begin{thmx}~\label{kth}
Let $(y_{n}(x))_{n\in \N}$ be a sequence of real-valued $C^{1}$ functions defined on an interval $[a,b]$. 
Suppose $y_{m}'(x)-y_{n}'(x)$ is monotone on $[a,b]$ for any two integers $m\not=n$ 
and suppose 
$$
\inf_{m\not=n}\min_{x\in [a,b]} |y_{m}'(x)-y_{n}'(x)| >0.
$$
Then for almost all $x\in [a,b]$, the sequence ${\bf y} = (y_{n}(x))_{n\in \N}$ is u.d.$1$.
\end{thmx}

\begin{proof}[Proof of Theorem~\ref{thmuni}]
Given a choice of $0\leq l_{1} <l_{2}< \cdots <l_{r}$ and $k_{i}\in \N$, $i=1, 2, \cdots, r$, consider the sequence 
 $$
\sum_{i=1}^{r} k_{i} \beta^{n+l_{i}}g(\beta) = \beta^{n}g(\beta) \sum_{i=1}^{r} k_{i} \beta^{l_{i}}=G(\beta) \beta^{n}
$$
where $G(\beta) =g(\beta) \sum_{i=1}^{r} k_{i} \beta^{l_{i}}\in C^{2}_{+}((0, \infty))$ since both $g(\beta)$ and $\sum_{i=1}^{r} k_{i} \beta^{l_{i}}$ are in $C^{2}_{+}((0, \infty))$.
Define a function 
$$
y_n(x)= G(x) x^n , \quad x\in (1, \infty).
$$ 

For $n>m$, we have
$$
y'_n(x)-y'_m(x)= G(x)(n x^{n-1}-m x^{m-1})+G'(x)(x^n-x^m)
$$
Since $G', G''\geq 0$ and since $n x^{n-m}-m\ge n-m\ge 1$ for $x> 1$,
we see that every term in the expression of $y'_n(x)-y'_m(x)$ are in $\mathcal{C}^{1}_{+}([a,\eta])$ 
for any $1<a<\eta<\infty$. 
By the closure property of $\mathcal{C}^{1}_{+}([a,\eta])$, we have that 
\begin{equation}~\label{inc}
y'_n(x)-y'_m(x)\in \mathcal{C}^{1}_{+}([a, \eta]), \quad \forall\; n>m. 
\end{equation}
In particular, this implies that $y'_n-y'_m$ is increasing for $n>m$.
Furthermore, we see that there is a constant $L>0$ such that
\begin{equation}~\label{min}
|y_n'(x)-y'_m(x)|\ge L, \quad \forall\; n>m\in \N, \;\; \forall\; a \leq x\leq \eta. 
\end{equation}
(\ref{inc}) and (\ref{min}) say that the sequence $(y_{n}(x))_{n\in \N}$ of real-valued $C^{1}$ functions satisfies all the hypotheses of Theorem~\ref{kth}. 

Theorem~\ref{kth} implies that for almost all $x$ in $[(2^{k}+1)/2^{k}, (2^{k-1}+1)/2^{k-1}]$ or $[k, k+1]$ for $k\geq 2$, the sequence 
$(y_n(x))_{n\in \N}$ is u.d.$1$. Further, this implies that for almost all 
$$
x\in (1, \infty)= \cup_{k=2}^{\infty} \Big[\frac{2^{k}+1}{2^{k}}, \frac{2^{k-1}+1}{2^{k-1}}\Big]\cup \cup_{k=2}^{\infty} [k, k+1]
$$ 
the sequence 
$(y_n(x))_{n\in \N}$ is u.d.$1$. For a fixed $\alpha >0$ or $\alpha <0$, using similar arguments, we have that for almost all $x\in (1, \infty)$, $(\alpha y_n(x))_{n\in \N}$ is u.d.$1$. 

Let 
$$
A_{r, (l_1,\cdots, l_r), (k_{1}, \cdots, k_{r})}=\{ \beta >1 \ | \ \alpha \beta^n g(\beta)
\sum_{i=1}^{r} k_{i} \beta^{l_{i}}  
\text{ is not u.d.$1$ } \}.
$$
Then the one dimensional Lebesgue measure of $A_{r, (l_1,\cdots, l_r), (k_{1}, \cdots, k_{r})}$ is zero. 
Since the set 
$$
U=\{(r, (l_1,\dots,l_r), (k_{1}, \cdots k_{r})) \ |\ r, k_1, \cdots, k_{r}, l_{2}<\cdots <l_{r}\in \N, l_{1}\in \N\cup\{0\} , l_{1}<l_{2}\}
$$ 
is countable, the one-dimensional Lebesgue measure of
$$
\bigcup_{(r, (l_1,\dots,l_r), (k_{1}, \cdots k_{r})) \in U} A_{r, (l_1,\cdots, l_r), (k_{1}, \cdots, k_{r})}
$$
is zero too.

For a fixed real number $\alpha\neq 0$ in the theorem, take a real number 
$$
\beta\in (1,\infty) \setminus \bigcup_{(r, (l_1,\dots,l_r), (k_{1}, \cdots k_{r})) \in U} A_{r, (l_1,\cdots, l_r), (k_{1}, \cdots, k_{r})}.
$$
This says that the sequence 
$$
\Big( \sum_{i=1}^{r} \alpha \beta^{n+l_{i}} g(\beta)\Big)_{n\in \N}
$$ 
is u.d.$1$ for all 
$$
(r, (l_1,\dots,l_r), (k_{1}, \cdots k_{r}))\in U.
$$
This completes the proof.
\end{proof}


\begin{thebibliography}{1}

\bibitem{A1} el Abdalaoui, e. H., Oscillating sequences, Gowers norms and Sarnak’s conjecture. arXiv:1704.07243v3.

\bibitem{Ab} Abramov, L. M., Metric automorphisms with quasi-discrete spectrum. Izv. Akad.
Nauk SSSR Ser. Mat. {\bf 26} (1962), 513-530.

\bibitem{AJ} Akiyama, S. and Jiang, Y., Higher order oscillation and uniform distribution. Uniform Distribution Theory. Uniform Distribution Theory, Volume 14 (2019), no. 1, 1-10. 

\bibitem{BKM} Byszewski, J., Konieczny, J., and M\"illner, C., Gowers Norms for Automatic Sequences. arXiv:2002.09509.

\bibitem{C}  Chowla, S., The Riemann hypothesis and Hilbert’s tenth problem. Mathematics and Its Applications, Vol {\bf 4}, Gordon and Breach Science Publishers, New York, 1965.

\bibitem{DJ}  Dai, X. and Jiang, Y., Distance entropy of dynamical systems on noncompact phase spaces. Discrete and Continuous Dynamical Systems, Vol. {\bf 20}, No. 2 (2008), 313-333.

\bibitem{Da} Davenport, H., {\em On some infinite series involving arithmetical functions (II)}. Quart. J. Math. Oxford, {\bf 8} (1937), 313-320.

\bibitem{Fo} Fomin, S., On dynamical systems with a purely point spectrum. Dokl. Akad. Nauk SSSR {\bf 77} (1951), 29-32 (in Russian).

\bibitem{FJ} Fan, A. and Jiang, Y.. {\em Oscillating Sequences, MMA and MMLS Flows and Sarnak's Conjecture}. Ergodic Theory \& Dynamical. Systems, August 2018, Vol. 38, no. 5, 1709–1744.

\bibitem{HM} Haase, M. and Moriakov, N., On systems with quasi-discrete spectrum. arXiv:1509.08961v3.

\bibitem{HP} Hahn, F. and Parry, W., Minimal dynamical systems with quasi-discrete spectrum.
J. London Math. Soc. {\bf 40} (1965), 309-323.

\bibitem{Hua} Hua, L.,  {\it Additive Theory of Prime Numbers} (Translations of Mathematical Monographs : Vol. {\bf 13}). American Mathematical Society. 1966.

\bibitem{JPAMS} Jiang, Y., Orders of oscillation motivated by Sarnak's conjecture. Proceedings of AMS, Volume {\bf 147}, Number 7, July 2019, 3075-3085.

\bibitem{K} Kronecker, L., Zwei S\"atze \"uber Gleichungen mit ganzzahligen Coefficienten. J. Reine Angew. Math., 53 (1857), pp. 173-175 (in German).

\bibitem{KN} Kuipers, L. and Niederreiter, H., Uniform Distribution of Sequences. J. Wiley and Sons, New York, 1974.

\bibitem{LTY} Li, J., Tu, S., and Ye, X., {\em Mean-equicontinuity and mean sensitivity}, Ergod. Th. \& Dynam. Sys., {\bf 35} (2015), 2587-2612.

\bibitem{Sa1} Sarnak, P., Three lectures on the M\"{o}bius function, randomness and dynamics. IAS Lecture Notes, 2009;\\ http://publications.ias.edu/sites/default/files/MobiusFunctionsLectures(2).pdf.

\bibitem{Sa2} Sarnak, P., M\"{o}bius randomness and dynamics. Not. S. Afr. Math. Soc. {\bf 43} (2012), 89-97.

\bibitem{Si} Sinai, Y., On the concept of entropy for a dynamical system. Dokl. Akad. Nauk SSSR, 124 (1959), pp. 768-771

 \end{thebibliography}
\end{document}